\documentclass[11pt, a4paper]{article}
\usepackage{graphicx}
\usepackage{amssymb,amsmath,amsthm}
\usepackage{verbatim}
\usepackage{color}
\usepackage[shortlabels]{enumitem}
\usepackage{dsfont}
\usepackage[margin=2.75cm]{geometry}
\usepackage{thmtools, thm-restate}
\usepackage{hyperref}
\usepackage[capitalise]{cleveref}

\declaretheorem{theorem}
\declaretheorem[sibling=theorem]{corollary, lemma, proposition, question, definition, conjecture}
\declaretheorem[numbered=no,name=Question]{quest}
\declaretheorem[numbered=no]{fact}
\declaretheorem[numbered=no,name=Expander Mixing Lemma]{mixingLemma}

\def\F{\mathbb{F}}

\title{Almost spanning distance trees in subsets of finite vector spaces}
\author{Debsoumya Chakraborti\thanks{
Mathematics Institute, University of Warwick, Coventry, UK.
Supported by the Institute for Basic Science (IBS-R029-C1), and the European Research Council (ERC) under the European Union Horizon 2020 research and innovation programme (grant agreement No. 947978).
E-mail: {\tt debsoumya.chakraborti@warwick.ac.uk}.}
\and 
Ben Lund\thanks{
Discrete Mathematics Group (DIMAG), Institute for Basic Science (IBS), Daejeon, South Korea.
Supported by the Institute for Basic Science (IBS-R029-C1).
E-mail: {\tt benlund@ibs.re.kr}.}
}

\newcounter{propcounter}

\begin{document}

\maketitle

\begin{abstract}
For $d\ge 2$ and an odd prime power $q$, consider the vector space $\mathbb{F}_q^d$ over the finite field $\mathbb{F}_q$, where the distance between two points $(x_1,\ldots,x_d)$ and $(y_1,\ldots,y_d)$ is defined as $\sum_{i=1}^d (x_i-y_i)^2$.
A distance graph is a graph associated with a non-zero distance to each of its edges.
We show that large subsets of vector spaces over finite fields contain every nearly spanning distance tree with bounded degree in each distance.
This quantitatively improves results by Bennett, Chapman, Covert, Hart, Iosevich, and Pakianathan on finding distance paths, and results by Pham, Senger, Tait, and Thu on finding distance trees.
A key ingredient in proving our main result is to obtain a colorful generalization of a classical result of Haxell about finding nearly spanning bounded-degree trees in an expander. 
\end{abstract}

\medskip

\begin{footnotesize}
\end{footnotesize}

\medskip

\section{Introduction}

Throughout this paper, $q$ denotes an odd prime power. 
For $d \geq 2$, $\F_q^d$ is the $d$-dimensional vector space over the field $\F_q$ with $q$ elements.
For $x = (x_1,\ldots,x_d) \in \F_q^d$, we define
${\|x\| = x_1^2 + \ldots + x_d^2,}$
and for $x,y\in \F_q^d$, refer to $\|x-y\|$ as the {\em distance} between $x$ and $y$.

An influential result of Iosevich and Rudnev \cite{iosevich2007erdos} is that, if $S \subseteq \F_q^d$ with $|S| > 2q^{(d+1)/2}$, then $\{\|x - y\|: x,y \in S\} = \F_q$.
The question of determining the largest set in $\F_q^d$ that avoids some distance is analogous to the Falconer distance problem in fractal geometry \cite{du2023weighted, du2019sharp, falconer1985hausdorff, guth2020falconer}, and related to the Erd\H{o}s unit and distinct distance problems in discrete geometry \cite{erdos1946sets, garibaldi2011erdos,guth2015erdHos,zahl2019breaking}.

Following the work of Iosevich and Rudnev, the finite field Falconer distance problem and its variants have been extensively studied.
Hart, Iosevich, Koh, and Rudnev showed \cite{hart2011averages} that the exponent $(d+1)/2$ in the bound $|S| > 2q^{(d+1)/2}$ is best possible for odd dimensions $d$.
It is conjectured that this exponent can be improved to $d/2$ for even $d$.
There has been progress on a variant of this problem in which the distance set is required to be a positive proportion of $\F_q$, rather than all of $\F_q$.
For example, Chapman, Erdogan, Hart, Iosevich, and Koh \cite{chapman2012pinned} showed that, if $S \subseteq \F_q^2$ with $|S| = \Omega(q^{4/3})$, then $\#\{\|x-y\|: x,y \in S\} = \Omega(q)$.
This result was strengthened by Hanson, Lund, and Roche-Newton \cite{hanson2016distinct}, who showed that the same hypothesis implies that there exists a point $y \in S$ such that $\#\{\|x-y\|: x \in S\} = \Omega(q)$.
These results were further strengthened for prime fields in \cite{lund2020bisectors, murphy2022pinned}.
The current record is by Murphy, Petridis, Pham, Rudnev, and Stevens \cite{murphy2022pinned}, who showed that if $S \subseteq \F_p^2$ for prime $p$ and $|S| = \Omega(p^{5/4})$, then there exists a point $y \in S$ such that $\#\{\|x-y\|: x \in S\} = \Omega(p)$.
In the other direction, Murphy and Petridis \cite{murphy2019example} found an infinite family of examples of subsets of $\mathbb{F}_q^2$ of size $q^{4/3}$ whose distance sets are not the whole of $\F_q$.

In this paper, we are interested in embedding more complex graphs than single edges.
To make this precise, a {\em distance graph} $\mathcal{H}$ is a graph $H$ together with a function $f:E(H) \rightarrow \F_q^*$ that associates each edge of $H$ with a non-zero distance.
A set $S \subseteq \F_q^d$ {\em contains} the distance graph $\mathcal{H}$ if there is an injective map $\phi:V(H) \rightarrow S$ such that $\|\phi(u)-\phi(v)\| = f(uv)$ for each edge $uv \in E(H)$. Such a map $\phi$ is often referred to as an \textit{embedding}, and we say there is an embedding of $\mathcal{H}$ in $S$.
A general question is: 
\begin{quest}
What is the minimum $n$ such that every set $S\subseteq \F_q^d$ of size $n$ must contain every member of some specified family of distance graphs?
\end{quest}
This question has been studied for complete graphs \cite{parshall2017simplices}, bounded-degree graphs \cite{iosevich2019embedding}, cycles \cite{iosevich2021cycles, pham2022geometric}, paths \cite{bennett2016long}, rectangles \cite{lyall2022weak}, and trees \cite{pham2022geometric,soukup2019embeddings}.
All of these previous works consider point configurations that are small relative to the size of the set that we are embedding the configurations into.
For example, Bennett, Chapman, Covert, Hart, Iosevich, and Pakianathan~\cite{bennett2016long} showed that $S\subseteq \F_q^d$ contains every distance path of length $O(|S|q^{-(d+1)/2})$.

By contrast, a special case of our main result is that every $S\subseteq \F_q^d$ contains every bounded-degree distance tree on $|S|-\Omega(q^{(d+2)/2})$ vertices. 
Informally, we show that sufficiently large point sets contain every nearly spanning distance tree with bounded degree.
Our main results are stated formally in the following subsection.

\subsection{Embedding distance trees}

For the convenience of stating our results, we introduce some notations.
For a set $R \subseteq \F_q^*$, an {\em $R$-distance graph} $\mathcal{H}$ is a graph $H$ with an associated function $f: E(H) \rightarrow R$ that assigns a distance to each edge of $H$. 
When $|R| = 1$, we refer to $\mathcal{H}$ as a {\em single-distance graph}, and when $R=\F_q^*$ we refer to $\mathcal{H}$ as a {\em distance graph}.
Consider a vertex $v\in V(H)$.
The neighborhood of $v$ is denoted by $\Gamma_H(v)$, and the degree of $v$ is $D_H(v) = |\Gamma_H(v)|$.
For $r \in R$, the {\em $r$-neighborhood} of $v$ is $\Gamma_{\mathcal{H},r}(v) = \{u \in \Gamma_H(v): f(uv) = r\}$, and the {\em $r$-degree} of $v$ is $D_{\mathcal{H},r}(v) = |\Gamma_{\mathcal{H},r}(v)|$.
The maximum degree of $H$ is $\Delta(H) = \max_{v \in V(H)}D_H(v)$, and the minimum degree of $H$ is $\delta(H) = \min_{v \in V(H)}D_H(v)$.
The maximum $r$-degree of $\mathcal{H}$ is $\Delta_r(\mathcal{H}) = \max_{v \in V(H)}D_{\mathcal{H},r}(v)$.
We often omit subscripts when they are clear from the context.
The following is our main result.
\begin{restatable}{theorem}{distanceTrees}\label{thm:distance trees}
    Let $R \subseteq \F_q^*$, let $\Delta \geq 2$ be an integer, and denote $t = |R|$.
    Every set $S \subseteq \F_q^d$ of points contains every $R$-distance tree $\mathcal{T}$ with at most $|S|-30 (t\Delta)^{1/2}q^{(d+1)/2}$ vertices and maximum $r$-degree $\Delta_r(\mathcal{T}) \leq \Delta$ for each $r \in R$. 
\end{restatable}

For example, if $|S| = 120 q^{(d+2)/2}$, then $S$ contains every distance tree with at most $2^{-1}|S|$ vertices and maximum $r$-degree at most $4$ for each $r \in \F_q^*$. 

The assumption that $0 \notin R$ can be removed in odd dimensions without changing the conclusion of the theorem.
See the discussion following \cref{thm:singleDistanceExpansion} for a more detailed discussion of this issue.

The proof of \cref{thm:distance trees} relies on an independently interesting colorful generalization of a theorem of Haxell \cite{haxell2001tree} on embedding trees in expander graphs, which we state in the next subsection.

\cref{thm:distance trees} has the {error term} $\Omega(t^{1/2}\Delta^{1/2}q^{(d+1)/2})$.
Constructions in \cite{hart2011averages, iosevich2023quotient} show that the dependence on $q$ is necessary and, in odd dimensions, best possible.
In the remainder of this subsection, we discuss the dependence of this term on $\Delta$, $t$, and $|S|$. 

The dependence on $\Delta$ is best possible.
For $t=1$ and $|S| = \Omega(q^d)$,  \cref{thm:distance trees} implies that $S$ contains all single-distance trees with maximum degree $O(q^{d-1})$ on at most $|S|/2$ vertices.
This is clearly within a constant factor of being optimal since the total number of points in $\F_q^d$ at distance $r$ from any point is roughly $q^{d-1}$.
However, the following theorem shows that it is possible to find stars of much higher degree than is guaranteed by \cref{thm:distance trees}.

\begin{restatable}{proposition}{distanceStars}\label{thm:distance stars}
    Let $R \subseteq \F_q^*$, and denote $t = |R|$.
    For any $S \subseteq \F_q^d$ with $|S| \geq 12 t^{1/2}q^{(d+1)/2}$, there is a subset $W \subseteq S$ of size $|W| \geq |S| - 80 t q^{(d+1)}|S|^{-1}$ such that, for each $x \in W$ and $r \in R$, we have $\#\{y \in W: \|x - y\| = r\} \geq 6^{-1}q^{-1}|S|$.
\end{restatable}
In particular, each vertex of $W$ is the central vertex of a star with minimum degree $6^{-1}q^{-1}|S|$ in each distance.
For example, if $|S| = 30 q^{(d+2)/2}$, then $S$ contains all distance stars $\mathcal{T}$ such that for each $r \in \F_q^*$, we have $\Delta_r(\mathcal{T}) \leq 5q^{d/2}$, which is far more than what we get from \cref{thm:distance trees}.

The error term $\Omega(t^{1/2}\Delta^{1/2}q^{(d+1)/2})$ does not depend on $|S|$.
The authors believe that this error term should be smaller for larger $|S|$.
As evidence for this belief, we have the following theorem for single-distance trees, which is an easy application of a recent result of Han and Yang~\cite{han2022spanning} on embedding spanning trees into expander graphs.

\begin{restatable}{theorem}{singleDistanceTrees}\label{thm:single distance trees}
    For any $\Delta \geq 2$, there is a constant $C_\Delta$ such that the following holds.
    Every $S \subseteq \F_q^d$ with ${|S| \geq C_\Delta^{\sqrt{d\log q}} q^{(d+1)/2}}$ contains every single-distance tree $\mathcal{T}$ with at most ${|S| - 32 q^{d+1}|S|^{-1}}$ vertices and maximum degree $\Delta(\mathcal{T}) \leq \Delta$.
\end{restatable}

\cref{thm:single distance trees} shows that if $S \subseteq \F_q^d$ with $|S| = \Omega(q^{(d+1)/2 + o(1)})$, then $S$ contains all single-distance trees $\mathcal{T}$ with bounded maximum degree and $|S| - |V(\mathcal{T})| = \Omega(q^{d+1}|S|^{-1})$, which is smaller than $q^{(d+1)/2}$ when $|S|$ is much larger than $q^{(d+1)/2}$.
We provide a family of constructions in \cref{sec:constructions} showing that this error term of $\Omega(q^{d+1}|S|^{-1})$ cannot be improved in odd dimensions.

Because of the dependence of the error term on $t$, \cref{thm:distance trees} does not guarantee that all bounded-degree distance trees with at most $q$ vertices can be embedded into a set with fewer than $q^{(d+2)/2}$ points.
The authors do not believe that so many points are needed, and propose the following conjecture.

\begin{conjecture}\label{conj:optimal distance trees}
    For any $\Delta \geq 2$, there is a constant $C_\Delta$ such that the following holds.
    Every $S \subseteq \F_q^d$ with $|S| \geq C_\Delta^{o(d \log q)} q^{(d+1)/2}$ contains every distance tree $\mathcal{T}$ with at most $|S|-100q^{d+1}|S|^{-1}$ vertices and maximum $r$-degree $\Delta_r(\mathcal{T}) \leq \Delta$ for each $r \in \F_q^*$.
\end{conjecture}

Up to the factors of $C_\Delta^{o(d \log q)}$ and $100$, \cref{conj:optimal distance trees} is as strong as possible. 
It would be interesting to prove the following weaker version.

\begin{conjecture}[Weaker version of \cref{conj:optimal distance trees}] \label{conj:weak distance trees}
    For any $\Delta \geq 2$, there is a constant $C_\Delta$ such that the following holds.
    Every $S \subseteq \F_q^d$ with $|S| \geq C_\Delta^{o(d \log q)} q^{(d+1)/2}$ contains every distance tree $\mathcal{T}$ with at most $|S|-100q^{(d+1)/2}$ vertices and maximum degree $\Delta(\mathcal{T}) \leq \Delta$.
\end{conjecture}

\cref{conj:weak distance trees} differs from \cref{conj:optimal distance trees} in that the error term does not depend on $|S|$, and we only control the total degree of $\mathcal{T}$ instead of the individual $r$-degrees of $\mathcal{T}$.

There is a natural barrier in our proof of \cref{thm:distance trees} to removing the dependence of the error term on $t$.
We discuss this issue in more detail and propose paths toward proving \cref{conj:optimal distance trees} in the concluding remarks.

\subsection{Expanders and the proof of \cref{thm:distance trees}}\label{sec:proof sketch}

\cref{thm:distance trees} is a special case of a more general theorem on embedding edge-colored trees into induced subgraphs of spectral expanders. 
Here we introduce the ideas used in the proof of this more general statement, and show how it implies \cref{thm:distance trees}.

Friedman and Pippenger \cite{friedman1987expanding} introduced an idea to embed a tree vertex by vertex inside a large expander by maintaining certain invariants. Haxell \cite{haxell2001tree} later improved this idea to show that each member of a certain family of expander graphs contains every possible nearly-spanning tree with bounded degree.
The first step toward proving \cref{thm:distance trees} is to generalize Haxell's result to embed edge-colored trees in families of expander graphs.
We need a few notations to state it.

Consider a family $\mathcal{G}=\{G_1,\ldots,G_t\}$ of graphs on the same vertex set $V(\mathcal{G})$. 
We usually view the indices in $[t]$ as colors and think of the family $\mathcal{G}$ as an edge-colored graph with an edge $e$ being colored with $i\in [t]$ if $e\in E(G_i)$ (note that $e$ may have multiple colors).
For every $X\subseteq V(\mathcal{G})\times [t]$, define $\Gamma_{\mathcal{G}}(X) = \bigcup_{(v,r)\in X} \Gamma_{G_{r}}(v)$.
For a graph $H$ and an edge-coloring map $f:E(H)\rightarrow [t]$, denote the corresponding edge-colored graph by $\mathcal{H}$. 
We say that $\mathcal{G}$ contains $\mathcal{H}$ if there is an embedding $\phi:\mathcal{H}\hookrightarrow \mathcal{G}$ such that for all $e\in E(H)$, we have $\phi(e)\in G_{f(e)}$.
We often say that $\mathcal{G}$ contains \textit{every $[t]$-colored $H$} if for every coloring $f:E(H)\rightarrow [t]$, the family $\mathcal{G}$ contains $\mathcal{H}$.

\begin{restatable}[Colorful version of Haxell's theorem]{theorem}{colorHaxell} \label{thm:haxell colored}
    Let $\Delta,m,k,t \in \mathbb{N}$ and let $\mathcal{G}=\{G_1,\ldots,G_t\}$ be a family of graphs on the same vertex set $V(\mathcal{G})$. 
    Suppose
    \begin{enumerate}
        \item $|\Gamma_{\mathcal{G}}(X)| \geq \Delta |X| + 1$ for all $X \subseteq V(\mathcal{G})\times [t]$ with $1 \leq |X| \leq m$, and
        \item $|\Gamma_{\mathcal{G}}(X)| \geq \Delta|X| + k$ for all $X \subseteq V(\mathcal{G})\times [t]$ with $m < |X| \leq 2m$. \label{condition:large sets}
    \end{enumerate}
    Then, $\mathcal{G}$ contains every $[t]$-colored tree $\mathcal{T}$ with at most $k$ vertices and maximum $r$-degree $\Delta_r(\mathcal{T}) \leq \Delta$ for each $r \in [t]$.
\end{restatable}

Next, we show that \cref{thm:haxell colored} can be applied to a family induced by a large vertex-subset of a family of good spectral expanders.
A graph $G$ with adjacency matrix $A_G$ is an $(n,D,\lambda)$-graph if $G$ is a $D$-regular graph on $n$ vertices and all but at most one of the eigenvalues of $A_G$ are bounded in absolute value by $\lambda$.
Note that the largest eigenvalue of a $D$-regular graph is always $D$.

One important property of $(n,D,\lambda)$-graphs is given by the classical expander mixing lemma.
For a graph $G$ and sets $X,Y \subseteq V(G)$, denote
\[e_G(X,Y) = \#\{(x,y) \in X \times Y: xy \in E(G)\}. \]
Notice that $e_G(X,Y)$ counts edges in $X \cap Y$ twice. We often drop the subscript from $e_G(X,Y)$ when the graph $G$ is clear from the context.
\begin{mixingLemma}[\cite{alon1988explicit, haemers1980eigenvalue}] \label{lem:expanderMixing}
    If $G$ is an $(n,D,\lambda)$-graph, and $X,Y \subseteq V(G)$, then
    \begin{equation} \label{eq:expander mixing lemma}
    \left |e(X,Y) - Dn^{-1}|X|\,|Y|\,\right| \leq \lambda \sqrt{|X|\,|Y|}. 
    \end{equation} 
\end{mixingLemma}

Given a graph $G$ and set $S \subseteq V(G)$, denote by $G[S]$ the subgraph of $G$ induced by~$S$.
Similarly, for a family $\mathcal{G} = \{G_1, \ldots, G_t\}$ of graphs and $S \subset V(\mathcal{G})$, denote $\mathcal{G}[S] = \{G_1[S], \ldots, G_t[S]\}$.

In general, subgraphs induced by vertex-subset of a family of $(n,D,\lambda)$-graphs may not satisfy the hypotheses of \cref{thm:haxell colored}.
Indeed, if $G$ is a graph satisfying \eqref{eq:expander mixing lemma}, and $X \subseteq S \subseteq V(G)$ with $|S|\,|X| < \left(\frac{n \lambda}{D} \right)^2$, then every vertex in $X$ may be isolated in $G[S]$. 
However, we show that this is essentially the only thing that can go wrong.
If $S \subseteq V(\mathcal{G})$ is sufficiently large, then we can find a large set $W \subseteq S$ such that $\mathcal{G}[W]$ has large minimum degree.

\begin{restatable}{lemma}{largeMinDegree}\label{thm:subgraphsWithLargeMinDegreeExist}
    Let $t \in \mathbb{N}$ and $C\ge 4t^{1/2}$ be a real number.
    Let $\mathcal{G} = \{G_1, \ldots, G_t\}$ be a family of $(n,D,\lambda)$-graphs on the same vertex set $V(\mathcal{G})$.
    Let $S \subseteq V(\mathcal{G})$ with $|S| = C \frac{n\lambda}{D}$.
    Then, there is a subset $W \subseteq S$ with $|W| \geq (1-8tC^{-2})|S|$ such that $\delta(G_r[W]) \geq 4^{-1}C\lambda$ for each $r\in [t]$.
\end{restatable}

It turns out that $\mathcal{G}[W]$ for the $W$ found in \cref{thm:subgraphsWithLargeMinDegreeExist} satisfies the hypotheses of \cref{thm:haxell colored}, which leads to the following result.

\begin{restatable}{theorem}{applicationHaxell}\label{thm:coloredTreesInInducedSubgraphs}
    Let $t \in \mathbb{N}$, and let $\mathcal{G} = \{G_1,\ldots,G_t\}$ be a family of $(n,D,\lambda)$-graphs on the same vertex set $V(\mathcal{G})$.
    Let $\Delta \geq 2$ be an integer, and let $S \subseteq V(\mathcal{G})$.
    Then, $\mathcal{G}[S]$ contains every $[t]$-colored tree $\mathcal{T}$ with at most $|S| - 10 (t\Delta)^{1/2} \frac{n \lambda}{D}$ vertices and maximum $r$-degree $\Delta_r(\mathcal{T}) \leq \Delta$ for each $r \in [t]$.
\end{restatable}

In order to derive \cref{thm:distance trees} using \cref{thm:coloredTreesInInducedSubgraphs}, we need to associate the distances of $R \subseteq \F_q^*$ with a family of $(n,D,\lambda)$-graphs.
Denote by $G_r^d$ the graph with vertex set $V(G_r^d) = \F_q^d$ such that $xy \in E(G_r^d)$ if and only if $\|x-y\| = r$.
The following result was proved by Iosevich and Rudnev \cite{iosevich2007erdos}, and independently by Medrano, Myers, Stark, and Terras \cite{medrano1996finite}.
\begin{theorem}[\cite{iosevich2007erdos, medrano1996finite}]\label{thm:singleDistanceExpansion}
    If $r \neq 0$ , then $G_r^d$ is a $(q^d, q^{d-1}+O(q^{\frac{d-1}{2}}), 2q^{\frac{d-1}{2}})$-graph.
\end{theorem}

If $r=0$ and $d$ is odd, then $G_r^d$ is a $(q^d,q^{d-1},2q^{\frac{d-1}{2}})$-graph - for example, see Proposition 2.2 in \cite{koh2015distance}.
Consequently, for odd $d$, the assumption that $0 \notin R$ can be removed from \cref{thm:distance trees} without changing the conclusion.
If $r=0$ and $d$ is even, then $G_r^d$ is a $(q^d, q^{d-1} + O(q^{d/2}), q^{d/2})$-graph, and so the assumption that $0 \notin R$ is necessary for the proof of \cref{thm:distance trees} to work.

\cref{thm:singleDistanceExpansion} implies that $G_r^d$ is a $D$-regular graph, for some $D = q^{d-1}+O(q^{(d-1)/2})$.
The $O(q^{(d-1)/2})$ term depends on $d$, $q$, and the quadratic character of $r$.
For odd dimensions $d$, this term is always either $-q^{(d-1)/2}$ or $q^{(d-1)/2}$, and for even dimensions it is always either $-q^{(d-2)/2}$ or $q^{(d-2)/2}$.
We will often use the following bounds to state our results more simply: \[(2/3)q^{d-1} \leq D \leq (4/3)q^{d-1}.\]

With \cref{thm:singleDistanceExpansion} in hand,
\cref{thm:distance trees,thm:distance stars} are straightforward corollaries of \cref{thm:coloredTreesInInducedSubgraphs,thm:subgraphsWithLargeMinDegreeExist}, respectively.

\begin{proof}[Proof of \cref{thm:distance trees}]
    \cref{thm:distance trees} is the special case of \cref{thm:coloredTreesInInducedSubgraphs} applied to $\mathcal{G} = \{G_r^d: r \in R\}$.
    By \cref{thm:singleDistanceExpansion}, we may use $n=q^d$, $D \geq (2/3)q^{d-1}$, and $\lambda = 2q^{(d-1)/2}$.
\end{proof}

\begin{proof} [Proof of \cref{thm:distance stars}]
\cref{thm:distance stars} follows immediately from \cref{thm:subgraphsWithLargeMinDegreeExist} and \cref{thm:singleDistanceExpansion}.
\end{proof}

\subsection{Comparison with earlier work}

\cref{thm:distance trees} yields nearly-spanning distance trees, where previous work has focused on finding much smaller structures.
However, our work yields quantitative improvements to these earlier bounds even for relatively small trees and paths.

Pham, Senger, Tait, and Thu \cite{pham2022geometric} proved that, if $\mathcal{T}$ is a $[t]$-edge-colored tree with $t$ vertices, $\mathcal{G} = (G_1, \ldots, G_t)$ is a collection of $(n,D,\lambda)$ graphs, and $S \subseteq V(\mathcal{G})$ with $|S| \geq t^2 \frac{n\lambda}{D}$, then $\mathcal{G}[S]$ contains at least $|S|\,t^{-1} - \frac{n\lambda}{D}$ disjoint copies of $\mathcal{T}$.

Neglecting constants, \cref{thm:coloredTreesInInducedSubgraphs} implies a stronger result for smaller sets.
Observe that a tree on $t$ vertices uses at most $t-1$ distinct colors, and has maximum degree at most $t-1$.
\cref{thm:coloredTreesInInducedSubgraphs} implies that, if $|S| = \Omega(t \frac{n \lambda}{D})$, then $S$ contains every $[t]$-edge-colored tree with maximum degree $t$ on at most $|S|- \Omega(t \frac{n \lambda}{D})$ vertices.
In particular, $S$ contains every tree that can be constructed by taking $|S|t^{-1} - \Omega(\frac{n \lambda}{D})$ disjoint copies of $\mathcal{T}$, and connecting them without introducing new colors or increasing the maximum degree.

Bennett, Chapman, Covert, Hart, Iosevich, and Pakianathan \cite{bennett2016long} proved that if $S \subseteq \F_q^d$ with $|S| \geq 4t(\ln 2)^{-1}q^{(d+1)/2}$, then $S$ contains every distance path of length $t$.
\cref{thm:arbitraryPaths} (which is a special case of \cref{thm:distance trees} with a simpler proof) implies that, for $R \subseteq \F_q^*$ with $|R|=t$, $S$ contains every $R$-distance path of length less than $|S|-4t^{1/2}q^{(d+1)/2}$.
We can use the fact that $|\F_q^*| = q-1$ to obtain the result that $S$ contains every distance path of length $|S| - 4q^{(d+2)/2}$.

\subsection{Organization of the paper}

To complete the proof of \cref{thm:distance trees}, we still need to prove \cref{thm:haxell colored,thm:subgraphsWithLargeMinDegreeExist,thm:coloredTreesInInducedSubgraphs}. However, before doing that, in the next section, we give a simpler proof of \cref{thm:distance trees} for the case of paths, without using \cref{thm:haxell colored}. In \cref{sec:haxell colored}, we prove \cref{thm:haxell colored}, and in \cref{sec:induced subgraphs}, we prove \cref{thm:subgraphsWithLargeMinDegreeExist,thm:coloredTreesInInducedSubgraphs}, thus finally completing the proof of the ingredients of \cref{thm:distance trees}.
\cref{sec:induced subgraphs} also contains the proof of \cref{thm:single distance trees}. In \cref{sec:constructions}, we provide constructions showing that \cref{thm:single distance trees,conj:optimal distance trees} are essentially tight up to constant factors. We end with a discussion on possible approaches to prove \cref{conj:optimal distance trees}, and offer a more general conjecture on families of expanders.

\section{Finding paths with depth-first search}

In this section, we use a modified depth-first search algorithm to embed paths with multiple distances.
This yields a simple proof of the special case of \cref{thm:distance trees} for embedding distance paths.
Depth-first search is a standard tool to find long paths in graphs, for instance, see Section~7 in Krivelevich's survey on expander graphs \cite{krivelevich_2019}.

We include two variations on the proof, which both rely on the same algorithm.
The first proof uses \cref{thm:singleDistanceExpansion}.
The second proof yields a slightly less general result, but relies on a point-sphere incidence bound instead of \cref{thm:singleDistanceExpansion}.
Although this point-sphere incidence bound is tight in general, it may be possible to improve it in the particular case we use.
Together with our proof, this would lead to progress on \cref{conj:optimal distance trees}.
We discuss this potential approach in more detail in the concluding remarks. 

The input to our algorithm is a collection $\mathcal{G}=\{G_1, \ldots, G_t\}$ of graphs on the same vertex set $V(\mathcal{G})$ and a $[t]$-colored path $\mathcal{P}=(v_1,\ldots,v_{\ell})$ that we intend to embed in $\mathcal{G}$. 
We maintain sets of vertices $A,B_1,\ldots,B_t$ and an ordered list $U$, which are updated as the algorithm proceeds.
Initially, $A = V(\mathcal{G})$ and $U,B_1,\ldots,B_t$ are all empty.
The algorithm terminates when $A$ and $U$ are both empty.

Each round of the algorithm either moves a vertex from $A$ to $U$, or from $U$ to one of the sets $B_1, \ldots, B_t$, as follows.
If $U$ is empty, move a vertex from $A$ to $U$.
Otherwise, suppose that $U=(u_1,\ldots,u_k)$, and let $r_k$ be the color of the edge $v_k v_{k+1}$ in the path $\mathcal{P}$. 
If there is a vertex $w \in A$ such that $u_k w \in E(G_{r_k})$, then move $w$ from $A$ to the end of $U$ (setting $u_{k+1}=w$).
Otherwise, move $u_k$ to $B_{r_k}$.
The key properties maintained at each step of this algorithm are as follows:
\begin{enumerate}[leftmargin=*]
    \item $\mathcal{G}[U]$ always contains a copy of the edge-colored path induced by the first $|U|$ vertices of~$\mathcal{P}$. 
    \item For each $r \in [t]$, we have $e_{G_r}(A,B_r) = 0$.
\end{enumerate}
The first property follows from the fact that whenever we add a vertex to $U$, it has an edge of the appropriate color to the previous vertex in $U$.
The second property follows from the fact that a vertex is only added to $B_r$ if it does not have any neighbor in $A$ in $G_{r}$.

Both proofs in this section proceed by showing that, for certain choices of numbers $a,b$, if $|A| = a$, then $\sum_{r \in [t]}|B_{r}| \leq b$.
Consequently, at the step of the algorithm when $|A|=a$, we have $|U| \geq |V(\mathcal{G})| - (a+b)$, and hence it is possible to embed $\mathcal{P}$ into $\mathcal{G}$ as long as the number $\ell$ of vertices in $\mathcal{P}$ satisfies $\ell\le |V(\mathcal{G})| - (a+b)$. 

\begin{theorem}\label{thm:arbitraryPaths}
    Let $\mathcal{G} = \{G_1, \ldots, G_t\}$ be a family of $(n,D,\lambda)$-graphs on the same vertex set $V(\mathcal{G})$, and let $S \subseteq V(\mathcal{G})$.
    Then, $\mathcal{G}[S]$ contains every $[t]$-colored path $\mathcal{P}$ with at most ${|S| - 2  \lambda n D^{-1} t^{1/2}}$ vertices.
\end{theorem}
\begin{proof}
    Run the modified depth-first search algorithm until the first iteration at which \linebreak ${|A| < 1 + \lambda n D^{-1}t^{1/2}}$. 
    Let $r \in [t]$ be arbitrary. Since $e_{G_r}(A,B_r) = 0$, the expander mixing lemma implies that
    \[D n^{-1}|A|\,|B_r| \leq \lambda |A|^{1/2}|B_r|^{1/2}. \]
    Rearranging, we have $|B_r| \leq \lambda n D^{-1}t^{-1/2}$.
    Hence, 
    $\sum_{r \in [t]} |B_r| \leq \lambda n D^{-1}t^{1/2}$,
    which completes the proof. 
\end{proof}

Together with \cref{thm:singleDistanceExpansion}, \cref{thm:arbitraryPaths} immediately yields the following special case of \cref{thm:distance trees}.

\begin{corollary}\label{thm:distance paths}
    Let $t \in \mathbb{N}$ with $0 < t < q$, and let $R \subseteq \F_q^*$ with $|R|=t$.
    If $S \subseteq \F_q^d$, then $S$ contains every $R$-distance path $\mathcal{P}$ on at most $|S|-4t^{1/2}q^{(d+1)/2}$ vertices.
\end{corollary}

Our second proof relies on a bound on the number of incidences between a set of points and a set of spheres in $\F_q^d$.
A \textit{sphere} with radius $r\in \F_q$ and center at $y\in \F_q^d$ is defined to be the set $\{x:\|x-y\|=r\}$.
The number of incidences between a set $X$ of points and a set $Y$ of spheres is
\[I(X,Y) = \#\{(x,y) \in X \times Y: x \in y\}. \]
The following bound on the number of incidences between a set of points and a set of spheres was proved independently by Cilleruelo, Iosevich, Lund, Roche-Newton, and Rudnev \cite{cilleruelo2014elementary} and  by Phuong, Pham, and Vinh \cite{phuong2017incidences}.
\begin{theorem}[\cite{cilleruelo2014elementary, phuong2017incidences}]\label{thm:generalPointSphereBound}
    Let $X$ be a set of points and $Y$ a set of spheres in $\F_q^d$.
    Then,
    \[\left |I(X,Y) - q^{-1}|X|\,|Y|\right | \leq q^{d/2}|X|^{1/2}|Y|^{1/2}. \]
\end{theorem}

Instead of using \cref{thm:singleDistanceExpansion}, we can apply \cref{thm:generalPointSphereBound} to obtain an alternative proof of the following special case of \cref{thm:distance paths}.

\begin{theorem}\label{thm:specialPaths}
    If $S \subseteq \F_q^d$, then $S$ contains every distance path $\mathcal{P}$ with at most ${|S| - 2q^{(d+2)/2}}$ vertices. 
\end{theorem}
\begin{proof}
    Let $R:=\F_q^*$.
    Run the modified depth-first search algorithm until the first time when $|A| < 1 + q^{(d+2)/2}$.
    For each $r \in R$, let $C_r$ be the set of spheres of radius $r$ and centers in $B_r$, and let $\mathcal{C} = \bigcup_{r \in R}C_r$.
    Note that $I(A,\mathcal{C}) = 0$.
    Hence, \cref{thm:generalPointSphereBound} implies
    \[q^{-1}|A|\,|\mathcal{C}| \leq q^{d/2}|A|^{1/2}|\mathcal{C}|^{1/2}, \]
    and hence $\sum_{r \in R} |B_r| = |\mathcal{C}| \leq q^{(d+2)/2}$.
    This completes the proof.
\end{proof}

\section{Colorful version of Haxell's result}\label{sec:haxell colored}

Our proof of \cref{thm:haxell colored} modifies and extends the arguments in \cite{draganic2022rolling,friedman1987expanding,haxell2001tree}.
\colorHaxell*
\begin{proof}
For every $X\subseteq V(\mathcal{G})\times [t]$, denote $\Gamma(X)=\Gamma_{\mathcal{G}}(X)$. 
For a set $X\subseteq V(\mathcal{G})\times [t]$, an edge-colored graph $\mathcal{H}$ with $\Delta_r(\mathcal{H})\le \Delta$ for each $r\in [t]$, and an embedding $\phi:\mathcal{H} \hookrightarrow \mathcal{G}$, define 
\[
R(X,\phi) = |\Gamma(X)\setminus \phi(\mathcal{H})| - \sum_{(v,r)\in X} (\Delta - D_{\mathcal{H},r}(\phi^{-1}(v))), 
\]
where we abuse the notation to mean $D_{\mathcal{H},r}(\phi^{-1}(v))=0$ when $v\notin \phi(\mathcal{H})$. 
An embedding $\phi:\mathcal{H} \hookrightarrow \mathcal{G}$ is called \textit{$s$-good} if, for each $X \subseteq V(\mathcal{G})\times [t]$ with $|X| \leq s$, we have $R(X,\phi)\ge 0$.
We will use the following.
\begin{fact}
    Let $\mathcal{H}$ be an edge-colored graph with at most $k-1$ vertices and $\Delta_r(\mathcal{H})\le \Delta$ for every $r\in [t]$.  
    Suppose $\phi:\mathcal{H} \hookrightarrow \mathcal{G}$ is a $2m$-good embedding.
    Then, we have the following.
    \stepcounter{propcounter}
    \begin{enumerate}[leftmargin=*,label = {\bfseries \emph{\Alph{propcounter}\arabic{enumi}}}]
    \item \label{item:small sets} If $X\subseteq V(\mathcal{G})\times [t]$ with $|X|\le 2m$ satisfies $R(X,\phi)=0$, then $|X|\le m$. 
    \item \label{item:submodularity} Every $X_1,X_2\subseteq V(\mathcal{G})\times [t]$ satisfy $R(X_1\cup X_2,\phi) + R(X_1\cap X_2,\phi) \le R(X_1,\phi) + R(X_2,\phi)$.     
    \item \label{item:turn to union} If $X_1,X_2\subseteq V(\mathcal{G})\times [t]$ with $|X_1|,|X_2|\le m$ satisfy $R(X_1,\phi) = R(X_2,\phi) = 0$. Then, $R(X_1\cup X_2,\phi) = 0$ and $|X_1\cup X_2|\le m$. 
    \end{enumerate}
\end{fact}

\begin{proof}
Let $X\subseteq V(\mathcal{G})\times [t]$ with $|X|\le 2m$ and $R(X,\phi)=0$. Since $\mathcal{H}$ has at most $k-1$ vertices, $|\Gamma(X)\setminus \phi(\mathcal{H})|\ge |\Gamma(X)|-k+1$. Clearly, ${\sum_{(v,r)\in X} (\Delta - D_{\mathcal{H},r}(\phi^{-1}(v)))\le \Delta |X|}$. Since $R(X,\phi)=0$, we have $|\Gamma(X)|\le \Delta |X| + k -1$. Thus, by \eqref{condition:large sets}, $|X|\le m$. This proves~\ref{item:small sets}.

Let $X_1,X_2\subseteq V(\mathcal{G})\times [t]$. The inequality in \ref{item:submodularity} follows from the following observation.
\[
|\Gamma(X_1\cup X_2)\setminus \phi(\mathcal{H})| + |\Gamma(X_1\cap X_2)\setminus \phi(\mathcal{H})| \le |\Gamma(X_1)\setminus \phi(\mathcal{H})| + |\Gamma(X_2)\setminus \phi(\mathcal{H})|.
\]

Let $X_1,X_2\subseteq V(\mathcal{G})\times [t]$ with $|X_1|,|X_2|\le m$ and $R(X_1,\phi) = R(X_2,\phi) = 0$. Since $\phi$ is $2m$-good, $R(X_1\cup X_2,\phi), R(X_1\cap X_2,\phi)\ge 0$. Thus, by \ref{item:submodularity}, it holds that $R(X_1\cup X_2,\phi) = 0$. Since $|X_1\cup X_2|\le 2m$, by \ref{item:small sets}, we have $|X_1\cup X_2|\le m$. This proves \ref{item:turn to union}.
\end{proof}

To prove \cref{thm:haxell colored}, we show the stronger statement that for every $[t]$-colored tree $\mathcal{T}$ with at most $k$ vertices and $\Delta_r(\mathcal{T})\le \Delta$ for each $r\in [t]$, there is a $2m$-good embedding of $\mathcal{T}$ in $\mathcal{G}$.
We use induction on the number of vertices of $\mathcal{T}$.
The base case is easy to verify, i.e., to show that we can have a $2m$-good embedding of a single vertex in $\mathcal{G}$. 
Indeed, for any embedding $\phi$ of a single vertex tree to $\mathcal{G}$, consider $X\subseteq V(\mathcal{G})\times [t]$ with $|X|\le 2m$, then $R(X,\phi) \ge (\Delta |X| + 1 - 1) - \Delta |X| = 0$.
The induction step can be done by applying the following lemma.
\begin{lemma}\label{lem:single step}
    Let $H$ be a graph with $|V(H)|\le k-1$. Let $f:E(H)\rightarrow [t]$ be an edge-coloring such that
    the corresponding edge-colored graph $\mathcal{H}$ satisfies $\Delta_r(\mathcal{H})\le \Delta$ for every $r\in [t]$. Suppose $\phi:\mathcal{H} \hookrightarrow \mathcal{G}$ is a $2m$-good embedding. Let $H_*$ be a graph that can be obtained by adding a vertex of degree $1$ to $H$. Let $f_*:E(H_*)\rightarrow [t]$ be an edge-coloring of $H_*$ such that $f_*$ restricted to the graph $H$ is the same as $f$, and the corresponding edge-colored graph $\mathcal{H}_*$ satisfies $\Delta_r(\mathcal{H}_*)\le \Delta$ for each $r \in [t]$. Then, there is a $2m$-good embedding $\phi_*:\mathcal{H}_* \hookrightarrow \mathcal{G}$ that extends $\phi$.
\end{lemma}
\begin{proof}
Suppose the graph $H_*$ is obtained by adding a leaf $v$ to a vertex $u$ in $H$. 
Without loss of generality, assume that $f_*(uv)=1$, i.e., the color of the new edge $uv$ is $1$, and thus, $D_{\mathcal{H},1}(u)\le \Delta -1$. 
Let $W:=\Gamma_{G_1}(\phi(u))\setminus \phi(\mathcal{H})$. For every $w\in W$, we can extend the embedding $\phi$ to $\phi_w:\mathcal{H}_* \hookrightarrow \mathcal{G}$ by letting $\phi_w(v)=w$.
Thus, it is enough to show that there is $w\in W$ such that $\phi_w$ is $2m$-good.
For the sake of contradiction, suppose that for every $w\in W$, there exist $X_w\subseteq V(\mathcal{G})\times [t]$ with $|X_w|\le 2m$ such that $R(X_w,\phi_w)< 0$. 
Then,
\[
|\Gamma(X_w)\setminus \phi(\mathcal{H})| - |\Gamma(X_w)\setminus (\phi(\mathcal{H})\cup \{w\})| = \mathds{1}[w\in \Gamma(X_w)], \text{ and}
\]
\begin{align*}
\sum_{(x,r)\in X_w} \left(\Delta - D_{\mathcal{H},r}(\phi^{-1}(x))\right) - &\sum_{(x,r)\in X_w} \left(\Delta - D_{{\mathcal{H}_*,r}}(\phi_w^{-1}(x))\right) \\ 
&= \sum_{(x,r) \in X_w} \left( D_{\mathcal{H}_*,r}(\phi_w^{-1}(x)) - D_{\mathcal{H},r}(\phi^{-1}(x)) \right) \\
&= \mathds{1}[(w,1)\in X_w] + \mathds{1}[(\phi(u),1) \in X_w].
\end{align*}
Since $\phi$ is $2m$-good, for every $w\in W$, we have $R(X_w,\phi)\ge 0$. From the above two equalities and the assumption that $R(X_w,\phi_w)<0$, we have
\[
0 < R(X_w,\phi) - R(X_w,\phi_w) = \mathds{1}[w\in \Gamma(X_w)] - \mathds{1}[(w,1)\in X_w] - \mathds{1}[(\phi(u),1) \in X_w].
\]
Thus, for every $w\in W$, we must have $R(X_w,\phi) = 0$, $w\in \Gamma(X_w)$, $(w,1)\notin X_w$, and $(\phi(u),1) \notin X_w$. 

We next show that the set $X:=\bigcup_{w\in W}X_w$ has size at most $m$ and $R(X,\phi) = 0$. Indeed, this is trivial if $W=\emptyset$; otherwise, we can apply induction and \ref{item:turn to union} to deduce it. Next, consider the set $X_* := X\cup \{(\phi(u),1)\}$. Since $w\in \Gamma(X_w)$ for every $w\in W$, we have $W=\Gamma_{G_1}(\phi(u))\setminus \phi(\mathcal{H}) \subseteq \Gamma(X)$. Hence, 
\[
|\Gamma(X)\setminus \phi(\mathcal{H})| = |\Gamma(X_*)\setminus \phi(\mathcal{H})|.
\]
Since $(\phi(u),1) \notin X_w$ for all $w\in W$, we have $(\phi(u),1) \notin \bigcup_{w\in W} X_w = X$. This fact together with the above equation implies that 
\[
R(X_*,\phi)= R(X,\phi) - (\Delta - D_{\mathcal{H},1}(u)) <0,
\]
where the last step uses $R(X,\phi)=0$ and $D_{\mathcal{H},1}(u)\le \Delta -1$.
Since $|X_*| \le m+1 \le 2m$, the fact $R(X_*,\phi) < 0$ contradicts the assumption that $\phi$ is $2m$-good.
This proves \cref{lem:single step}.
\end{proof}
This finishes the proof of \cref{thm:haxell colored}.
\end{proof}

\section{Induced subgraphs of {\boldmath $(n,D,\lambda)$}-graphs}\label{sec:induced subgraphs}

In this section, we prove \cref{thm:subgraphsWithLargeMinDegreeExist}, \cref{thm:coloredTreesInInducedSubgraphs}, and \cref{thm:single distance trees}.
The proof of \cref{thm:coloredTreesInInducedSubgraphs} depends on \cref{thm:haxell colored} and \cref{thm:subgraphsWithLargeMinDegreeExist}.
\cref{thm:single distance trees} follows easily from \cref{thm:subgraphsWithLargeMinDegreeExist} and a recent theorem of Han and Yang \cite{han2022spanning}.
\cref{thm:subgraphsWithLargeMinDegreeExist} is of central importance in this section, and we restate it for convenience.
\largeMinDegree*

\begin{proof}
    Let $X_0 := \emptyset$. For $i\ge 1$, inductively construct the sets $X_i$ in the following way: If there is a vertex $v_i\in S\setminus X_{i-1}$ such that there exists $r\in [t]$ with $e_{G_r}(\{v_i\}, S \setminus X_{i-1}) < 4^{-1}C\lambda$, then $X_i := X_{i-1} \cup \{v_i\}$, otherwise, $X_i := X_{i-1}$.
    Let $j := 8tC^{-2}|S|$, and let $Y := X_j$. If $|Y|<j$, then there is nothing to prove since $W$ can be taken as $S\setminus Y$. Thus, we can assume that $|Y|=j$.
    Note that for every vertex $v\in Y$, there exists $r\in [t]$ such that ${e_{G_r}(\{v\}, S \setminus Y) < 4^{-1}C\lambda}$ for each $r \in [t]$.
    
    Let $r \in [t]$ be such that there is $Y_r\subseteq Y$ with $|Y_r| \geq |Y|/t = 8C^{-2}|S|$ and for every $v\in Y_r$, we have ${e_{G_r}(\{v\}, S \setminus Y) < 4^{-1}C\lambda}$.
    Thus, we have ${e_{G_r}(Y_r, S \setminus Y) < 4^{-1}C\lambda |Y_r|}$.
    Note that $|S \setminus Y| = (1-8tC^{-2})|S|$.
    By the expander mixing lemma and the assumption that $C\ge 4t^{1/2}$,
    \begin{align*}
    4^{-1}C\lambda|Y_r| &> e_{G_r}(Y_r,S \setminus Y) \geq \frac{D}{n}|Y_r|\,|S \setminus Y| - \lambda \sqrt{|Y_r|\,|S\setminus Y|}, \text{ hence} \\
    4^{-1}C &> (1-8tC^{-2})C - (1-8tC^{-2})^{1/2}(8C^{-2})^{-1/2}, \text{ and so}\\
    4^{-1} &> (1-8tC^{-2})^{1/2} \left((1-8tC^{-2})^{1/2} - 8^{-1/2}\right) \ge 4^{-1},
    \end{align*}
    which is a contradiction.
    This finishes the proof of \cref{thm:subgraphsWithLargeMinDegreeExist}.
\end{proof}

We next use \cref{thm:haxell colored,thm:subgraphsWithLargeMinDegreeExist} to prove \cref{thm:coloredTreesInInducedSubgraphs}.

\applicationHaxell*
\begin{proof}
    Denote $C = |S|\left(\frac{n \lambda}{D} \right)^{-1}$.
    Since the conclusion of the theorem is trivial if $C \leq 10 (t \Delta)^{1/2}$, we assume that $C > 10 (t \Delta)^{1/2}$.
    Thus, we can apply \cref{thm:subgraphsWithLargeMinDegreeExist} to find a subset $W \subseteq S$ such that $|W|\ge (1-8tC^{-2})|S|$ and $\delta(G_r[W]) \geq 4^{-1}C\lambda$ for each $r \in [t]$.
    We will show that $\mathcal{G}[W]$ satisfies the hypotheses of \cref{thm:haxell colored} with $k=|S|-10(t\Delta)^{1/2}\frac{n \lambda}{D}$, and $m=t^{1/2}\Delta^{-1/2}\frac{n\lambda}{D}$.
    
    For $X \subseteq W \times [t]$ and $r \in [t]$, denote $X_r = \{v \in W: (v,r) \in X\}$.
    To apply \cref{thm:haxell colored}, it will be sufficient to show the following for every $r\in [t]$:
    \begin{enumerate}
        \item\label{eq:expansion of small set} for $Z \subseteq W$ with $1 \leq |Z| \leq t^{-1}m$, we have $|\Gamma_{G_r[W]}(Z)| > t\Delta|Z|$, and
        \item\label{eq:expansion of big set} for $Z \subseteq W$ with $t^{-1}m < |Z| \leq 2t^{-1}m$, we have $|\Gamma_{G_r[W]}(Z)| > t\Delta|Z| + k$.
    \end{enumerate}
    Indeed, for any $X \subseteq W \times [t]$, there is some $r \in [t]$ such that $|X_r| \geq t^{-1}|X|$. Thus, we can take a subset $Z\subseteq X_r$ such that $|Z|=t^{-1}|X|$.
    Since $|\Gamma_\mathcal{G[W]}(X)| \geq |\Gamma_{G_r[W]}(Z)|$, the above conditions immediately imply the first and second numbered hypotheses of \cref{thm:haxell colored}, respectively. 

    We start with showing \eqref{eq:expansion of small set}.
    Assume for contradiction that there is some $Z\subseteq W$ with $1 \leq |Z| \leq t^{-1}m$ such that $|\Gamma_{G_r[W]}(Z)| \leq t \Delta |Z|$ for some $r\in [t]$.
    With the previously established bound $\delta(G_r[W]) \geq 4^{-1}C\lambda$, the expander mixing lemma, and the assumption that $|Z| \leq t^{-1}m = t^{-1/2}\Delta^{-1/2}\frac{n\lambda}{D}$, this implies
    \begin{align*}
    4^{-1}C\lambda |Z| \leq e(Z,\Gamma_{G_r[W]}(Z)) 
    &\leq \frac{D}{n}|Z|\,|\Gamma_{G_r[W]}(Z)| + \lambda \sqrt{|Z|\,|\Gamma_{G_r[W]}(Z)|} \\
    &\leq \frac{D}{n} \Delta m |Z| + \lambda |Z|t^{1/2} \Delta^{1/2} \\
    &= 2  t^{1/2} \Delta^{1/2} \lambda |Z|.
    \end{align*}
    Hence, $4^{-1}C \leq 2 t^{1/2}\Delta^{1/2}$, which contradicts the assumption that $C > 10t^{1/2}\Delta^{1/2}$.

    We next show \eqref{eq:expansion of big set}. 
    Suppose $Z \subseteq W$ with $t^{-1}m < |Z| \leq 2t^{-1}m$ and $r\in [t]$. 
    Let $Y = W \setminus \Gamma_{G_r[W]}(Z)$.
    Again using the expander mixing lemma,
    \[0 = e_{G_r[W]}(Z,Y) \geq \frac{D}{n}|Z|\,|Y| - \lambda \sqrt{|Z|\,|Y|}. \]
    Rearranging and using the assumption that $|Z| > t^{-1}m = (t \Delta)^{-1/2}\frac{n\lambda}{D}$,
    \[|Y| \leq \left ( \frac{n\lambda}{D} \right )^2 |Z|^{-1} < \frac{n \lambda}{D}(t \Delta)^{1/2}. \]
    Using $|Z| \leq 2t^{-1} m = 2 (t \Delta)^{-1/2} \frac{n\lambda}{D}$, we have
    \[t \Delta |Z| \leq 2 (t \Delta)^{1/2} \frac{n \lambda}{D}. \]
    Hence, using the assumption that $C > 10(t \Delta)^{1/2}$,
    \begin{align*}
    |\Gamma_{G_r[W]}(Z)| = |W| - |Y| 
    > (1- 8tC^{-2})|S| - (t \Delta)^{1/2} \frac{n \lambda}{D}  
    &\geq |S| - 2(t \Delta)^{1/2} \frac{n \lambda}{D} \\ 
    &\geq t \Delta |Z| + |S| - 10 (t \Delta)^{1/2} \frac{n \lambda}{D},
    \end{align*}
    as claimed.

    We have now established that the hypotheses of \cref{thm:haxell colored} are satisfied for the claimed values of $k$ and $m$, and the conclusion of \cref{thm:coloredTreesInInducedSubgraphs} follows.
\end{proof}

The remainder of this section is devoted to proving \cref{thm:single distance trees}.
We use the following recent result of Han and Yang. They showed that every sufficiently good expander contains copies of all possible spanning trees of bounded maximum degree.
\begin{theorem}[Theorem~1.5, \cite{han2022spanning}]\label{thm:han}
    Let $n \in \mathbb{N}$ be sufficiently large, let $\Delta \in \mathbb{N}$, and let $p,\beta$ be positive real numbers, with
    \[\beta \leq \frac{pn}{4 \Delta^{5\sqrt{\log n}}}. \]
    Let $G$ be an $n$-vertex graph such that
    \stepcounter{propcounter}
    \begin{enumerate}[label = {\bfseries \emph{\Alph{propcounter}\arabic{enumi}}}]
        \item \label{eq:pseudorandom} $\left |e(X,Y) - p|X|\,|Y|\right | \leq \beta \sqrt{|X|\,|Y|}$ for all $X,Y \subseteq V(G)$, and
        \item \label{eq:min degree} $\delta(G) \geq 4 \sqrt{p\beta n}$.
    \end{enumerate}
    Then, $G$ contains a copy of every tree with maximum degree $\Delta$ on at most $n$ vertices.
\end{theorem}

\begin{theorem}\label{thm:treesInInducedSubgraphsHan}
    Let $G$ be an $(n,D,\lambda)$-graph with $n$ sufficiently large, let $C \ge 256$ be a real number, and let $S \subseteq V(G)$ with $|S| = C \frac{n\lambda}{D}$.
    Then, $G[S]$ contains a copy of every tree $T$ with at most $(1-8C^{-2})|S|$ vertices and $\Delta(T) \leq (5^{-1}C)^{(5 \sqrt{\log(|S|)})^{-1}}$.
\end{theorem}
\begin{proof}
    To ensure \ref{eq:min degree}, we remove the low degree vertices of $G[S]$ using \cref{thm:subgraphsWithLargeMinDegreeExist} as follows.
    By \cref{thm:subgraphsWithLargeMinDegreeExist}, there is a set $W \subseteq S$ with $|W| \geq (1-8C^{-2})|S|$ such that ${\delta(G[W]) \geq 4^{-1}C\lambda}$.
    It is also immediate from the expander mixing lemma that the induced subgraph $G[W]$ of the $(n,D,\lambda)$-graph $G$ satisfies \ref{eq:pseudorandom} with $p=Dn^{-1}$ and $\beta = \lambda$.
    With the same parameters and $n=|W|$, \ref{eq:min degree} holds because of the following inequality.
    \[\delta(G[W]) \geq 4^{-1}C\lambda \geq 4C^{1/2}\lambda \geq 4 \sqrt{Dn^{-1}\lambda |S|} \geq 4\sqrt{p\beta|W|}. \]
    The condition on $\Delta(T)$ implies the condition on $\beta$ in \cref{thm:han} with $\Delta = \Delta(T)$.  
    Finally, the proof of \cref{thm:treesInInducedSubgraphsHan} follows from applying \cref{thm:han} to $G[W]$ with parameters $p = Dn^{-1}$, $\beta = \lambda$, $n=|W|$, and $\Delta = \Delta(T)$.
\end{proof}

\singleDistanceTrees*

\begin{proof} 
\cref{thm:single distance trees} follows immediately from \cref{thm:singleDistanceExpansion,thm:treesInInducedSubgraphsHan}. 
\end{proof}

\section{Constructions}\label{sec:constructions}

In this section, we describe two families of constructions in odd dimensions.
Our constructions build on a construction of Iosevich, Koh, and Rakhmonov \cite{iosevich2023quotient}.
Consider the following immediate consequence of \cref{thm:singleDistanceExpansion} and the expander mixing lemma:
\begin{theorem}\label{thm:distancesBetweenTwoSets}
    Let $X,Y \subseteq \F_q^d$, and denote
    \[S_r(X,Y) = \{(x,y) \in X \times Y: \|x-y\| =r\}. \]
    If $r \neq 0$, then
    \[\left ||S_r(X,Y)| - (q^{-1} + O(q^{-(d+1)/2}))|X|\,|Y| \right | \leq 2 q^{(d-1)/2}|X|^{1/2}|Y|^{1/2}. \]
\end{theorem}
In particular, if $|S_r(X,X)| = 0$ for some $r \in \F_q^*$, then $|X| < 2q^{(d+1)/2}$. This special case of \cref{thm:distancesBetweenTwoSets} is shown by Iosevich and Rudnev \cite{iosevich2007erdos}.
The exponent $(d+1)/2$ is known to be best possible in odd dimensions; for example, see~\cite[Theorem~2.7]{hart2011averages}.

We describe two families of constructions that show that both the upper and lower bounds of \cref{thm:distancesBetweenTwoSets} are tight for many different sizes of $X,Y$.
In particular, the first construction shows that the bound $|S|-|V(\mathcal{T})|=\Omega(q^{d+1}|S|^{-1})$ in \cref{thm:single distance trees} cannot be improved in odd dimensions, even for large $S$.

It is more convenient to work with a quadratic form that is equivalent to distance, rather than working with distance directly.
Two quadratic forms $Q_1,Q_2$ over $\F_q^d$ are equivalent if there is a nonsingular linear transformation $M$ such that $Q_1(x) = Q_2(Mx)$.
None of the results concerning distance sets in this paper depend on the specific quadratic form used, only on the equivalence class.
A standard result on quadratic forms over finite fields of odd order (see {\em e.g.} \cite[Section~2]{iosevich2023quotient}) is that there are precisely two equivalence classes: either the determinant of the matrix associated to a quadratic form is square, or it is non-square.
In particular, the determinant of the quadratic form $\|x\| = x_1^2 + x_2^2 + \ldots + x_d^2$ is $1$, a square.

We assume that $d$ is odd, and work with the quadratic form
\begin{equation}\label{eq:quadraticForm}Q(x) = x_1^2 - x_2^2 + \ldots +x_{d-2}^2 - x_{d-1}^2 + \mu x_d^2, \end{equation}
where $\mu$ is chosen so that $Q$ is equivalent to $\|\cdot\|$.
In more detail, the determinant of $Q$ is $\mu$ if $d \equiv 1 \mod 4$, and is $-\mu$ if $d \equiv 3 \mod 4$.
Since $Q$ is equivalent to $\|\cdot\|$, the determinant of $Q$ is square.
Denote by $\eta(x)$ the Legendre symbol of $x$, which takes value $1$ at a non-zero square, $-1$ at a non-square, and $\eta(0)=0$.
Then, $\eta(\mu) = 1$ if $d \equiv 1 \mod 4$, and $\eta(-\mu) = 1$ if $d \equiv 3 \mod 4$.
Since $-1$ is a square if and only if $q \equiv 1 \mod 4$, we have that $\eta(\mu)=-1$ if $d \equiv 3 \mod 4$ and $q \equiv 3 \mod 4$, and otherwise $\eta(\mu) = 1$.

Note that \cref{thm:distancesBetweenTwoSets} implies that, if $|S_r(X,Y)| = 0$ for some $r \in \F_q^*$, then ${|X|\,|Y| = O(q^{d+1})}$.
The next construction gives sets of various sizes that attain this upper bound.

\begin{proposition}\label{thm:zeroDistanceConstruction}
    Let $d,k \in \mathbb{N}$ with $d \geq 3$ odd and $1 \leq k < (d-1)/2$.
    For any $r \in \F_q^*$, there are sets $X,Y \subseteq \F_q^d$ such that $|X|= \Omega(q^{d-k})$ and $|Y|=q^{k+1}-q^k$, and $|S_r(P,Q)| = 0$.
\end{proposition}
\begin{proof}
    Let $R^\times = \{r - s: \eta(s) = \eta(\mu)\}$.
    Note that $|R^\times| = (q-1)/2$.
    Let
    \begin{align*}
        X &= \{(a_1,a_1,\ldots,a_k,a_k,x_1,\ldots,x_{d-2k}) \in \F_q^d: x_1^2 - x_2^2 + \ldots - x_{d-2k-1}^2 \notin R^\times\}, \text{ and} \\
        Y &= \{(b_1,b_1,\ldots,b_k,b_k,0,\ldots,0, y) \in \F_q^d: y \neq 0\}.
    \end{align*}
    Since $|\F_q \setminus R^\times| = \Omega(q)$, $|X| = \Omega(q^k q^{d-2k}) = \Omega(q^{d-k})$.
    It is easy to see that $|Y|=q^{k+1}-q^k$.
    For $(x,y) \in X \times Y$, all of the terms in $Q(x-y)$ that depend on the first $2k$ coordinates cancel, and
    \[Q(x-y) = x_1^2-x_2^2 + \ldots - x_{d-2k-1}^2 + \mu(x_{d-2k} - y)^2. \]
    Let $t_1 = x_1^2-x_2^2 + \ldots - x_{d-2k-1}^2$ and $t_2 = \mu(x_{d-2k} - y)^2$.
    Note that $\eta(t_2) = \eta(\mu)$.
    Since $t_1 \notin R^{\times}$, we have that $\eta(r-t_1) \neq \eta(\mu)$, and so $r-t_1-t_2 \neq 0$ and $Q(x-y) = t_1 + t_2 \neq r$.
\end{proof}

Iosevich, Koh, and Rakmonov used a similar analysis in \cite[Section 4]{iosevich2023quotient}.
In their construction, $X =\{(a_1, a_1, \ldots, a_{(d-1)/2}, a_{(d-1)/2}, b) \in \F_q^d\}$. 
For any $x,y \in X$, we have $Q(x-y) = (x_d - y_d)^2$, so $\mu(Q(x-y)) = 1$.
In particular, if $\mu(r)=-1$, then $S_r(X,X) = 0$.

\cref{thm:single distance trees} states that $S$ contains every bounded-degree $r$-distance tree $\mathcal{T}$ on at most $|S|-\Omega(q^{d+1}|S|^{-1})$ vertices.
\cref{thm:zeroDistanceConstruction} shows that, for many sizes $|S|$, there are sets $S$ and $Y$ with $|Y| = \Omega(q^{d+1}|S|^{-1})$ such that there is no pair $(x,y) \in S \times Y$ with $\|x-y\| = r$.
Consequently, error term in \cref{thm:single distance trees} is within a constant factor of being as small as possible.

Another consequence of \cref{thm:distancesBetweenTwoSets} is that, if $X,Y \subseteq \F_q^d$ such that $\|x-y\| = r$ for each pair $(x,y) \in X \times Y$, then $|X|\,|Y| = O(q^{d-1})$.
The next construction shows that there are sets $X,Y$ of various sizes that are within a constant factor of attaining this bound.

\begin{proposition}
    Let $d,k \in \mathbb{N}$ with $d \geq 3$ odd and $1\leq k \leq (d-1)/2$.
    There are sets $X, Y \subseteq \F_q^d$ and $r \in \F_q^*$ such that $|X| =\Omega(q^{d-k-1})$, $|Y|=q^{k}$, and $|S_r(X,Y)| = |X|\,|Y|$.
\end{proposition}
\begin{proof}
    Let
    \begin{align*}
        X &= \{(a_1,a_1, \ldots, a_k,a_k,x_1, \ldots, x_{d-2k}) \in \F_q^d: x_1^2 - x_2^2 + \ldots - x_{d-2k-1}^2 + \mu x_{d-2k}^2 = r\} , \text{ and} \\
        Y &= \{(b_1,b_1,\ldots,b_k,b_k,0,\ldots,0) \in \F_q^d\}.
    \end{align*}
    Using the fact that the sphere of radius $r$ in $\F_q^{d-2k}$ contains $\Omega(q^{d-2k-1})$ points, we have $|X| = \Omega(q^k q^{d-2k-1}) = \Omega(q^{d-k-1})$, and it is easy to see that $|Y| = q^k$.
    For $(x,y) \in X \times Y$, all of the terms in $Q(x-y)$ that depend on the first $2k$ coordinates cancel, and
    \[Q(x-y) = x_1^2 - x_2^2 + \ldots + \mu x_{d-2k}^2 = r. \]
    
    In the special case that $k = (d-1)/2$, for $(x,y) \in X \times Y$ we have $\|x - y\| = \mu(x_1 - y_1)^2$; hence, in this special case, $\eta(r) = \eta(\mu)$.
    For $k < (d-1)/2$, the described construction works for arbitrary $r \in \F_q^*$.
\end{proof}

\section{Concluding remarks}\label{sec:discussConjecture}

\cref{conj:optimal distance trees} says that every sufficiently large set $S \subseteq \F_q^d$ contains every bounded-degree distance tree on $|S| - \Omega(q^{d+1}|S|^{-1})$ vertices.
By contrast, \cref{thm:distance trees} implies that every sufficiently large set $S\subseteq \F_q^d$ contains every bounded-degree distance tree on $|S| - \Omega(q^{(d+2)/2})$ vertices.
In this section, we will discuss where in our proof of \cref{thm:distance trees} the apparently sub-optimal error term comes from, as well as potential routes to improve it.

\cref{thm:distance trees} is an immediate corollary of \cref{thm:coloredTreesInInducedSubgraphs}.
The proof of \cref{thm:coloredTreesInInducedSubgraphs} proceeds in two steps.
First, we use \cref{thm:subgraphsWithLargeMinDegreeExist} to remove from $\mathcal{G}[S]$ all vertices that have low degree in any color.
Then, we apply our colorful generalization of Haxell's theorem to find a nearly spanning bounded-degree tree in what remains.

When we apply \cref{thm:subgraphsWithLargeMinDegreeExist}, we remove $O\left(t \left( \frac{n \lambda}{D} \right)^2 |S|^{-1}\right)$ vertices of $S$, leaving a set $W$ such that $\mathcal{G}[W]$ has large minimum degree in every color.
The factor of $t$ in this expression means that we cannot use \cref{thm:subgraphsWithLargeMinDegreeExist} in a proof of \cref{conj:optimal distance trees}.
Unfortunately, \cref{thm:subgraphsWithLargeMinDegreeExist} is tight.
For example, let $G_1$ be an $(n,D,\lambda)$-graph on vertex set $V(\mathcal{G})$, and let $Y_1 \subseteq S \subseteq V(\mathcal{G})$ so that $|\Gamma_{G_1[S]}(Y_1)| = 0$ and $|Y_1|=\Omega\left(\left(\frac{n\lambda}{D} \right)^2 |S|^{-1}\right)$.
Let $\phi_1$ be the identity map on $V(\mathcal{G})$, and let $\phi_2, \ldots, \phi_t: V(\mathcal{G})\rightarrow V(\mathcal{G})$ be permutations so that $\phi_i(S) = S$ for all $i \in [t]$ and $\phi_i(Y_1) \cap \phi_j(Y_1) = \emptyset$ for $1 \leq i < j \leq t$.
For $i \in [t]\setminus \{1\}$, let $G_i$ be the graph on $V(\mathcal{G})$ so that $xy \in E(G_i)$ if and only if $\phi_i^{-1}(x)\phi_i^{-1}(y) \in E(G_1)$.
Note that for each $i \in [t]\setminus \{1\}$, the graph $G_i$ is isomorphic to $G_1$, and thus also an $(n,D,\lambda)$-graph.
Then $\mathcal{G}=(G_1, \ldots, G_t)$ satisfies the hypotheses of \cref{thm:subgraphsWithLargeMinDegreeExist}, and $|Y_1 \cup \cdots \cup Y_t| = t|Y_1| =  \Omega \left( t \left(\frac{n \lambda}{D}\right)^2|S|^{-1} \right)$.

After removing low degree vertices with \cref{thm:subgraphsWithLargeMinDegreeExist}, we use \cref{thm:haxell colored} to find an edge-colored tree in the remaining edge-colored graph $\mathcal{G}[W]$.
This step introduces the dominant $\Omega(t^{1/2}\frac{n \lambda}{D})$ error term in \cref{thm:coloredTreesInInducedSubgraphs}.
In order to apply \cref{thm:haxell colored}, we need $|\Gamma_\mathcal{G[W]}(X)| > \Delta|X|$ for every not-too-large set $X \subseteq W \times [t]$.
The way that we ensure this in the proof of \cref{thm:coloredTreesInInducedSubgraphs} is by obtaining $|\Gamma_{G_r[W]}(Z)| > t\Delta|X|$ for every small $Z \subseteq W$ and $r \in [t]$.
By pigeonholing, there must be some $r \in [t]$ that accounts for at least a $t^{-1}$ fraction of $X$.
It is this pigeonholing argument that introduces the dependence of the error term on $t$.

We mention two natural directions for future research that might lead to improving \cref{thm:distance trees}.
The first is to prove a stronger version of \cref{thm:coloredTreesInInducedSubgraphs} that avoids the use of \cref{thm:subgraphsWithLargeMinDegreeExist} and \cref{thm:haxell colored}.
The second is to use the fact that the various distance graphs on $\F_q^d$ are closely related to each other to prove \cref{conj:optimal distance trees} without relying on \cref{thm:coloredTreesInInducedSubgraphs}.

As mentioned above, neither \cref{thm:subgraphsWithLargeMinDegreeExist} nor our application of \cref{thm:haxell colored} are adequate to prove \cref{conj:optimal distance trees}.
However, we conjecture the following improvement of \cref{thm:coloredTreesInInducedSubgraphs}, which would imply \cref{conj:optimal distance trees}.

\begin{conjecture}[Stronger version of \cref{thm:coloredTreesInInducedSubgraphs}]
    For any $\Delta \geq 2$, there is a constant $C_{\Delta}$ such that the following holds.
    Let $t \in \mathbb{N}$, and let $\mathcal{G} = \{G_1, \ldots, G_t\}$ be a family of $(n, D, \lambda)$-graphs on the same vertex set $V(\mathcal{G})$.
    If $S \subseteq V(\mathcal{G})$ with $|S| \geq C_\Delta^{o(\log (n))} \frac{n \lambda}{D}$, then $S$ contains every $[t]$-colored tree $\mathcal{T}$ with at most $|S| - 100 \left(\frac{n \lambda}{D} \right)^2|S|^{-1}$ vertices and maximum $r$-degree $\Delta_r(\mathcal{T}) \leq \Delta$ for each $r \in [t]$.
\end{conjecture}

The second idea is to use the structure of distances in $\F_q^d$ to prove \cref{conj:optimal distance trees} instead of relying on \cref{thm:coloredTreesInInducedSubgraphs}.
The example given above showing that \cref{thm:subgraphsWithLargeMinDegreeExist} cannot be improved is very different than distance graphs over $\F_q^d$, which are related to each other by dilations.
We conjecture that it is possible to improve \cref{thm:subgraphsWithLargeMinDegreeExist} in the special case when $G_1,\ldots,G_t$ are the single-distance graphs over $\F_q^d$.

\begin{conjecture}\label{conj:highDegreeInducedDistanceGraphs}
    Let $C > 50$, and let $S \subseteq \F_q^d$ with $|S|=Cq^{(d+1)/2}$.
    Then, there is a subset $W \subseteq S$ with $|W| \geq (1-100C^{-2})|S|$ such that, for each $x \in W$ and $r \in \F_q^*$, there are $\#\{y \in W: \|x - y\|=r\} \geq 10^{-1}q^{-1}|S|$ points of $W$ at distance $r$ from $x$.
\end{conjecture}

Proving \cref{conj:highDegreeInducedDistanceGraphs} would not immediately resolve \cref{conj:optimal distance trees}, but would provide some additional understanding of what separates finite field distance graphs from arbitrary families of expanders.

A related idea for using the specific structure of distance graphs in $\F_q^d$ is to use point-sphere incidence bounds, as in the proof of \cref{thm:specialPaths}.
\cref{thm:generalPointSphereBound} is tight in general, but has been improved in some special cases \cite{koh2022finite, koh2022point}.
For example, Koh, Lee, and Pham proved the following: 
\begin{theorem}[Theorem 1.6, \cite{koh2022finite}]\label{thm:kohLeePhamIncidence}
    Let $d \equiv 2 \mod 4$ and $q \equiv 3 \mod 4$.
    If $X \subseteq \F_q^d$ is a set of points and $Y$ is a set of spheres with $|Y| \leq q^{d/2}$, then
    \[\left |I(X,Y) - q^{-1}|X|\,|Y| \right |  = O\left (q^{(d-1)/2}|X|^{1/2}|Y|^{1/2} \right ). \]
\end{theorem}
Unfortunately, the constraint on $|Y|$ in \cref{thm:kohLeePhamIncidence} is too stringent for our application.
However, Koh, Lee, and Pham suggested that it may be possible to relax the condition to $|Y| \leq q^{(d+2)/2}$.
By the proof of \cref{thm:specialPaths}, such an improvement (with suitable constants) would immediately imply that, for $d \equiv 2 \mod 4$ and $q \equiv 3 \mod 4$, every set $S \subseteq \F_q^d$ of at least $|S| = 100q^{(d+1)/2}$ points contains every distance path on at least $2^{-1}|S|$ vertices.

\bibliographystyle{plain}
\bibliography{distanceTrees}

\end{document}